\documentclass[11pt, reqno]{amsart}
\usepackage{amsmath,amsfonts,amssymb,amsthm}
\usepackage{epsfig}

\usepackage[frenchb,english]{babel}
\usepackage{bbm}
\usepackage{enumerate}
\usepackage{amstext}
\usepackage{xspace}
\usepackage{amsfonts}
\usepackage{amsmath}
\usepackage{amssymb}
\usepackage{amstext}
\usepackage{amsthm}       
\usepackage{xspace}
\usepackage[active]{srcltx}

\newcommand{\ind}{\{1,\dots,d\}}

\newcommand{\R}{\mathbb R}
\renewcommand{\L}{\mathbb L}

\newcommand{\T}{\mathbb T}

\newcommand{\LL}{\mathcal{L}}

\newcommand{\eps}{\varepsilon}

\newcommand{\uu}{\mathbf u}
\newcommand{\vv}{\mathbf v}

\newcommand{\HH}{\mathbb H}

\newcommand{\1}{\mathbbm1}

\newcommand{\e}{\textrm{\rm e}}

\newcommand{\tagliato}{$\kern-5 mm -$}
\newcommand{\tagliat}{$\kern-4 mm -$}

\newcommand{\D}[1]{\mbox{\rm #1}}
\newcommand{\dd}{\D{d}}

\newtheorem{teorema}{Theorem}[section]
\newtheorem{prop}[teorema]{Proposition}

\newtheorem{definition}[teorema]{Definition}
\newtheorem{cor}[teorema]{Corollary}

\newtheorem{guess}[teorema]{Remark}
\newtheorem{example}[teorema]{Example}

\newenvironment{dimo}{{\bf\noindent Proof.}}{\qed}
\newenvironment{oss}{\begin{guess} \begin{rm}}{\end{rm} \end{guess}}
\newenvironment{definizione}{\begin{definition} \begin{rm}}{\end{rm}
\end{definition}}
\begin{document}

\title{Twisted Lax--Oleinik formulas and weakly coupled systems of Hamilton--Jacobi equations}
\author{Maxime Zavidovique}
\address{
IMJ (projet Analyse Alg\' ebrique), UPMC,  
4, place Jussieu, Case 247, 75252 Paris C\' edex 5, France}
\email{zavidovique@math.jussieu.fr} \keywords{weakly coupled systems of Hamilton--Jacobi equations, viscosity solutions, weak KAM Theory}
\subjclass[2010]{35F21, 49L25, 37J50.}

\maketitle

\selectlanguage{frenchb}

\begin{abstract}
Nous d\'  emontrons que les solutions de viscosit\' e d'un syst\` eme faiblement coupl\' e d'\' equations d'Hamilton--Jacobi peuvent--\^ etre approch\' ees par des it\' erations d'op\' erateurs tordus \` a la Lax--Oleinik. On \' etablit la convergence vers la solution du sch\' ema it\' eratif et mettons en exergue quelques propri\' et\' es suppl\' ementaires des solutions approch\' ees.
\end{abstract}

\selectlanguage{english}

\begin{abstract}
We show that  viscosity solutions of evolutionary weakly coupled systems of Hamilton--Jacobi equations can be approximated by iterated twisted Lax--Oleinik like operators.
We establish convergence  to the solution of the iterated scheme and discuss further properties of the approximate solutions.
 \end{abstract}
%
%
\section*{Introduction}
Representation formulas for solutions of Hamilton--Jacobi equations with Tonelli Hamiltonians are the starting point of important theories studying the qualitative properties of the PDE and of the associated dynamical system. Of course, we have in mind Fathi's weak KAM theory which builds a bridge between solutions of the stationary equation (or cell problem) and Aubry-Mather theory which deals with action minimizing trajectories and measures.

Establishing such a dual point of view has led to striking results, let us mention, out of many others, two of them: the longtime convergence of solutions of the Hamilton--Jacobi  equation (see for example \cite{Fa1, DS}) and the convergence of solutions to the discounted equations (\cite{DFIZ, IS}). For both of those examples, purely PDE proofs were later on found (for instance in \cite{ BS,BIM,CGMT,IMT} and references therein).

A natural generalization of Hamilton--Jacobi equations are systems of Hamilton--Jacobi equations and more particularly, weakly coupled systems, meaning that the coupling only appears on the $0$ order terms. Ironically, weak KAM theory for those systems evolved backward compared to what happened for a single equation. The study of the critical equation was done first, from a purely PDE angle in \cite{DavZav14}, before the dynamical aspects were highlighted (\cite{SMT14, ISZ}). Recently, Lax--Oleinik formulas, combined with a random framework were studied for evolutionary equation in \cite{DSZ}. However deterministic approaches had been tried previously without success. 

The goal of this paper is to take those deterministic formulas as a starting point and see how to recover the solutions of the weakly coupled system from them. We expect the reader to have some familiarity with viscosity solutions, see \cite{barles_book} for an introduction on the subject.

\subsection{Acknowledgement} The author wishes  to thank A. Davini, with whom he started thinking about this problem, for his insight and for enriching conversations. This research was financed by ANR WKBHJ (ANR-12-BS01-0020).

The author thanks the anonymous referee for his helpful advise in  improving the presentation of the paper.

\section{Setting and main result}
We will consider $d$ Lagrangians on $\T^N \times \R^N$. They will be denoted by $L_i, 1\leqslant i \leqslant d$.

Moreover, for technical reasons, we will make a couple of assumptions on the growth of the $L_i$ and their derivatives. 

\begin{definizione}
In the following $\theta : \R_+\to \R_+$ is a function (called Nagumo function) such that
\begin{equation}\label{nagumo}
\forall M>0,\ \ \exists K_M>0 ,\ \  \forall m\leqslant M,\ \  \forall q\geqslant 0,\quad \theta(q+m)\leqslant K_M\big(1+\theta(q)\big). \tag{N}
\end{equation}

We will say that a function  $L : [0,T]\times \T^N\times \R^N$ is a \emph{good Lagrangian} if it verifies the following set of hypotheses
\begin{itemize}
\item[(L1)]\label{L1} the Lagrangian $L$ is a $C^1$ function  and  for all $(t,x)\in [0,T] \times \T^N$, $L(t,x,\cdot)$ is a strictly convex function;
\item[(L2)]\label{L2}
there exists  constants $c_0>0$ and $A>0$  such that 
\begin{align*}
\forall (t,x,v)\in [0,T]\times\T^N\times \R^N,\quad& L(t,x,v) \geqslant \theta(|v|)-c_0;\\
& |\partial_x L(x,v)| +|\partial_v L(x,v)| < A \theta(|v|).
\end{align*}
\end{itemize}
\end{definizione}

{\bf We will hence assume that all the $L_i$, $1\leqslant i\leqslant d$ are good Lagangians} (with a common Nagumo function $\theta$ and constants $A$ and $c_0$).
\begin{oss}
This hypothesis is mainly technical and serves at one specific instance: Theorem \ref{CS} and its application in Proposition \ref{consistency}. It allows to obtain automatic Lipschitz estimates of minimizers of a minimization problem involving time--dependent Lagrangians. For autonomous Lagrangians, such hypotheses are not needed thanks to conservation of energy and Clarke--Vinter's theorem (\cite{ClVi85}) but we will have to deal with non--autonomous Lagrangians. 
\end{oss}

\begin{definition}\rm
A matrix $B\in \mathcal{M}_d(\R)$ is a coupling matrix if its non--diagonal entries are non--positive and the sum of the elements of each line is non negative.
\end{definition}

It follows from the above definition that the diagonal entries of $B$ verify $b_{ii}\geqslant 0$.
\vspace{.5cm}

We recall that given a Lagrangian $L$ on $\R_+\times\T^N \times \R^N$, such that each $L(t,\cdot,\cdot)$ verifies the above  hypotheses, its Hamiltonian is defined by 
$$\forall (t,x,p) \in \R_+\times \T^N \times \R^N, \quad H(t,x,p) = \sup_{v\in \R^N} \langle p,v\rangle - L(t,x,v).  $$
The Hamiltonian is then a strictly convex function of $p$, it is also superlinear.

In what follows, $H_i$ is the Hamiltonian associated to $L_i$;
\begin{definition}\label{solsys}\rm
Let $\uu^0 : \T^N\to \R^d$ be a continuous initial datum. The unique solution (see Proposition \ref{comparison}) to the evolutionary equation 
\begin{equation}\label{intro evo wcs}
\frac{\partial u_i}{\partial t}+H_i(x,D_xu_i)+\sum_{j=1}^m b_{ij} u_j(t,x)=0\quad\hbox{in $(0,+\infty)\times\T^N$,}\quad\hbox{$\forall i\in\ind$,} 
\end{equation}
with $\uu(0,\cdot) = \uu^0$ will be denoted by $(t,x)\mapsto S(t)\uu^0(x)$.
\end{definition}

\begin{oss}
Existence and uniqueness results for this equation are established in \cite{CamLey} under additional growth assumptions on the Hamiltonians. Those assumptions are removed in \cite[Proposition A.1]{DSZ}. The proofs follow the same path as for a single equation. First, a comparison principle is established (see Proposition \ref{comparison}). This uses in an essential manner the sign properties of the coupling matrix $B$ (they imply the system fall in a more general class of coupled systems, see \cite{Engler}). This comparison principle implies uniqueness and existence follows from Perron's method (properties of $B$ give that a supremum of subsolutions is a subsolution).
\end{oss}

\begin{definition}\rm
We will denote by $W(t)$ the twisted Lax--Oleinik formula which to a vector valued function $\uu^0: \T^N\to \R^d$ associates another vector valued function $W(t)\uu^0: \T^N\to \R^d $ the entries of which are, for $i\in \ind$:

$$[W(t)\uu^0 (x)]_i = \inf_{\substack{\gamma : [-t,0]\to \T^N \\ \gamma(0)=x }} \big[\e^{-tB}\uu^0\big(\gamma(-t)\big)\big]_i + \Big[\int_{-t}^0 \e^{sB} \L\big(\gamma(s),\dot\gamma(s)\big)\dd s\Big]_i,
$$

where the infimum is taken over all absolutely continuous curves $\gamma : [-t,0]\to \T^N$ and where $\L = (L_i)_{i\in \ind}$.

We will often use the following notation:
\begin{equation}\label{twistedLO}
W(t)\uu^0 (x) = \inf_{\substack{\gamma : [-t,0]\to \T^N \\ \gamma(0)=x }} \e^{-tB}\uu^0\big(\gamma(-t)\big) + \int_{-t}^0 \e^{sB} \L\big(\gamma(s),\dot\gamma(s)\big)\dd s.
\end{equation}
\end{definition}

Note that in the previous equation, we only write one infimum to have a synthetic formula, but there are actually $d$ quantities to minimize hence possibly $d$ different minimizing curves.

The goal of this note is to show a link between $W$ and $S$. As is easily seen, one reason why $W$ differs from $S$ (apart from the fact that it does not provide a solution of the weakly coupled system in any reasonable sense) is that it does not verify the semi--group property (or sometimes also referred to as dynamical programming property; an explicit counterexample is given in appendix).  At the contrary, by the uniqueness of viscosity solutions, $S$ is indeed a semi--group, meaning that for all $s,t\geqslant 0$ we have $S(t+s)=S(t)\circ S(s)$.

Our main result is the following:
\begin{teorema}\label{the}
Let $\uu:\T^N \to \R^d$ be a Lipschitz function, then for any $t\geqslant 0$, the following holds:
$$S(t)\uu = \lim_{n\to +\infty} W(t/2^n)^{2^n}\uu.$$
\end{teorema}
The procedure of considering iterates of $W$ is a natural way of forcing the semi--group property. It has already appeared, for example making a link between variational and viscosity solutions associated to non--convex Hamiltonians (in the case of a single equation). See the works of Wei (\cite{Wei}) and also of Roos (\cite{Roos}) for more details on this subject.

This can also be seen as a result on the convergence of an approximate scheme for the system. Many results in the literature of viscosity solutions justify that the result can be expected to be true (see \cite{BaSo}). We will give a self contained proof of our result in this particular setting which is an adaptation of the previous reference.  

Finally let us comment on the previous statement and its hypotheses. They are willingly stronger than necessary because the proof is more natural in this setting. However, the Lipschitz continuity of the initial  data can be weakened to continuity (Theorem \ref{C^0}). The regularity of the Lagrangians can also be lowered to Lipschitz  (Theorem \ref{L_ilip}). Finally the fact of taking increasing subdivisions of the interval $[0,t]$ is not necessary in order to obtain the convergence result. It is however more natural in some regards (see section \ref{further} and particularly Remark \ref{remarque}). 

\section{Preliminaries}
\subsection{About a single Hamilton--Jacobi equation}

Given a continuous function $u : \T^N\to \R$, let us define the Lax--Oleinik semi--group as follows: if $0\leqslant s<t$  and $x\in \T^N$ then
$$T_L^{t,s}u(x) = \inf_{\gamma(t-s)=x} u\big(\gamma (0)\big) + \int_{0}^{t-s} L\big(s+\sigma, \gamma(\sigma), \dot\gamma(\sigma)\big)\dd \sigma.$$
The infimum in the previous formula is taken amongst absolutely continuous curves $\gamma :[0,t-s] \to \T^N$. 
Clearly, this family of operators verifies a Markov property, meaning that if $0\leqslant s<t <t'$ then 
$T_L^{t',t}\circ T_L^{t,s}u = T_L^{t',s}$.

Let us recall hereafter some properties verified by such Lagrangian functions and their Lax--Oleinik semi--group. 

\begin{teorema}\label{CS}
Let $u : \T^N \to \R$ be a $K$-Lipschitz continuous function and $L : [0,T] \times \T^N \times \R^N$ be a good Lagrangian, we define  the function $U : [0,T] \times \T^N \to \R$  by $U(t,x) = T^{t,0}u(x)$.
\begin{enumerate}
\item The function $U$ is a viscosity solution to the Cauchy problem
\begin{equation}\label{Cauchy}
\begin{cases}
\partial_t U + H(t,x,D_x U) = 0,\\
U(0,\cdot) = u.
\end{cases}
\end{equation}
\item
For any $x\in \T^N$ and $0\leqslant s<t\leqslant T$, the infimum in the definition of the Lax--Oleinik semi--group is a minimum. Moreover, there exists a constant $M>0$ depending solely on $\theta$, $c_0$ and $K$ such that any minimizer $\gamma$ is $M$-Lipschitz and even $C^1$.
\item
Finally,  the function $U$ is  Lipschitz continuous (with  Lipschitz constants depending only on  $\theta$, $c_0$ and $K$)  in $[0,T]\times \T^N$ hence it is the unique viscosity solution to \eqref{Cauchy}.
\end{enumerate}
\end{teorema}

\begin{oss}
Lipschitz continuity of 
$U$ is a direct consequence of the Lipschitz continuity of the minimizing curves. Lipschitz continuity of minimizing curves is proved in  \cite[Theorems 6.2.5, 6.3.1]{CaSi00} from which the hypotheses on the Lagrangians are taken. The $C^1$ property of minimizing curves is a consequence of the strict convexity of the Lagrangians (see \cite[Ex. 18.5 p. 351]{Cl13}). Actually, finer properties can be obtained, as semi--concavity estimates (see for instance \cite[Theorems 6.4.2, 6.4.3 and 6.4.4]{CaSi00}) but we will not need them. The fact that existence of a Lipschitz solution to \eqref{Cauchy} implies uniqueness is a folklore result (see \cite{canson} and references therein or \cite[Proposition A.2]{DSZ}).
\end{oss}

\subsection{About systems}
Recall that
the matrix $B\in \mathcal{M}_d(\R)$ is a coupling matrix if its non--diagonal entries are non--positive and the sum of the elements of each line is non negative. We denote by $\1 = (1,\cdots , 1)^T$ the vector with all entries equal do $1$.

\begin{prop}\label{sto}
Under the above hypotheses, for all $s\geqslant 0$, the entries of $\e^{-sB}$ are all non--negative.
\end{prop}

\begin{proof}
This is an immediate consequence of the formula 
$$\e^{-sB} = \lim_{n\to +\infty}\Big( {\rm Id}_d -\frac{s}{n}B\Big)^n.$$
\end{proof}
It follows from the previous proposition and the fact that the exponential is smooth that:
\begin{prop}\label{good}
Let $T>0$, then for all $1\leqslant i \leqslant d$, the Lagrangian $\LL_i(s,x,v)= [\e^{sB}\L (x,v)]_i$ is a good Lagrangian on $[-T,0]\times \T^N\times \R^N$. 
\end{prop}

\begin{cor}\label{decroissant}
For all $s\geqslant 0$, the following inequalities hold:
$$0\1 \leqslant \e^{-sB}\1 \leqslant \1.$$
\end{cor}
\begin{proof}
The left inequality follows from Proposition \ref{sto}.
For the right inequality, write 
$$\frac{\dd}{\dd s}  \e^{-sB}\1 = \e^{-sB}(-B\1) \leqslant 0\1.$$
It follows that all entries are decreasing as $s$ increases. As equality holds for $s=0$ this proves the result.
\end{proof}

We now come back to the Definition \ref{solsys}. Such a solution exists and is unique thanks to  the following more general comparison principle (see \cite[Proposition 2.5]{DSZ}):

\begin{prop}\label{comparison}
Let $\overline u$ and $\underline u$ be respectively a lower semicontinuous supersolution and a bounded upper semicontinuous subsolution of \eqref{intro evo wcs}. Assume they are bounded on $[0,T]\times \T^N$, then $\overline u \geqslant \underline u$ on $[0,T]\times \T^N$.
\end{prop}

Let us now come back to the twisted Lax--Oleinik formula
$$
W(t)\uu^0 (x) = \inf_{\substack{\gamma : [-t,0]\to \T^N \\ \gamma(0)=x }} \e^{-tB}\uu^0\big(\gamma(-t)\big) + \int_{-t}^0 \e^{sB} \L\big(\gamma(s),\dot\gamma(s)\big)\dd s.
$$

Using the notation  $\LL_i(s,x,v)= [\e^{sB}\L (x,v)]_i$ the twisted Lax--Oleinik formula may be interpreted as follows: 

$$[W(t)\uu^0 (x)]_i = T_{\LL_i(\cdot-t,\cdot,\cdot)}^{t,0}[\e^{-tB}\uu^0]_i (x).$$

Hence, as by Proposition \ref{good}, the $\LL_i$ are good Lagrangians (when restricted to $t\in [-T,0]$), Theorem \ref{CS} applies to the twisted Lax--Oleinik formula.

\section{Proof of Theorem \ref{the}}

\begin{definizione}
Given $n>0$ and $t \in (0,T]$, let us define the iterated operator $W_n(t) = W(s)\circ \big(W(T/2^n)\big)^k$ where $s> 0$ and $k\geqslant 0$ are such that $t = kT/2^n + s $ and $s\leqslant T/2^n$.
\end{definizione}

Following \cite{BaSo}, we state some fundamental properties of the operators $W$ and $W_n$.

\begin{prop}
The operator $W$ verifies the following:
\begin{itemize}
\item \underline{Monotonous}: if $\uu\leqslant \vv$ and $t>0$ then $W(t)(\uu)\leqslant W(t)(\vv)$,
\item \underline{Continuity}: if $k\in \R$ then for any function $\uu$ and $t>0$, $W(t)(\uu + k\1) = W(t)(\uu) + k\e^{-tB}\1$. In particular, $W(t)$ is $1$--Lipschitz for the sup norm.
\end{itemize}
It follows immediately that $W_n$ enjoys the same properties.
\end{prop}
The second property follows from Corollary \ref{decroissant}. 

The last property we state is fundamental as it links the operators $W$ with the original system \eqref{intro evo wcs}:

\begin{prop}\label{consistency}
The operator $W$ is \underline{consistent} in the sense that if $\Phi :  \T^N \to \R^d$ is a $C^1$ function then
$$\lim_{t\to 0^+} \frac{W(t)\Phi - \Phi}{t} = -\HH(\cdot , D\Phi)-B\Phi,$$
where we use the notation $\HH(\cdot , D\Phi)=\big(H_i (\cdot , D \phi_i)\big)_{i\in \ind}$.
\end{prop}

\begin{proof}
Let us fix $x\in \T^N$ and $i\in \ind$. For $t\leqslant T$, let us denote by $\gamma_t : [-t,0] \to \T^N$ a curve realizing the minimum in \eqref{twistedLO} for the $i$-th equation. Recall that the curve $\gamma_t$ is then $C^1$. Moreover, for any $s\in [-t,0]$, the function $x\mapsto v(s,x) := T_{\LL_i(\cdot-t,\cdot,\cdot)}^{s,0}[\e^{-tB}\uu^0]_i (x)$ is differentiable at $\gamma_t(-t+s)$ and setting $p_t(-t+s)$ this differential, the couple $(\gamma_t,p_t)$ solves Hamilton's equations with Hamiltonian function $\mathcal H_i (\cdot-t,\cdot,\cdot)$ \big(associated to the Lagrangian $\LL_i(\cdot-t,\cdot,\cdot)$\big), see \cite[Theorem 6.3.3 and 6.4.7]{CaSi00}.

We may then compute 
\begin{align*}[W(t)\Phi - \Phi]_i (x) & = \Big[\e^{-tB}\Phi\big(\gamma_t(-t)\big) - \Phi\big(\gamma(0)\big) + \int_{-t}^0 \e^{sB} \L\big(\gamma_t(s),\dot\gamma_t(s)\big)\dd s\Big]_i\\
&=-\Big[\big(\Phi-\e^{-tB}\Phi\big)\big(\gamma(0)\big)\Big]_i \\
& \qquad \qquad -\int_{-t}^0\Big( \frac{\dd}{\dd s} v \big(s,\gamma_t(s)\big) 
 -\Big[ \e^{sB} \L\big(\gamma_t(s),\dot\gamma_t(s)\big)\dd s\Big]_i\Big) \dd s \\
&= - \Big[\big(\Phi-\e^{-tB}\Phi\big)(x)\Big]_i -\int_{-t}^{0} \mathcal H_i\big(s, \gamma_t(s),p_t(s)\big)\dd s \\\ &= [-tB\Phi(x)]_i - t \mathcal H_i(0 ,x, D_x\varphi_i) + t\varepsilon (t) \\&=  [-tB\Phi(x)]_i - t  H_i(0 ,x, D_x\varphi_i) + t\varepsilon (t),
\end{align*}
where $\varepsilon $ is a function going to $0$ as $t\to 0$.
Note that the function $\varepsilon $ depends on $t$, $x$, $i$, but due to the fact that the Lipschitz constant of $\gamma_t$ (and of $p_t$)  depends only on $\| D\Phi\|_\infty$ the convergence of $\varepsilon$ to $0$ is uniform with respect to $i$ and $x$.

This proves the proposition.

\end{proof}

As proved in \cite{BaSo}, consistency, monotonicity and continuity are enough to ensure that Theorem \ref{the} holds. For the sake of completeness, we reproduce the proof (adapted to our setting) below:

\begin{proof}[proof of Theorem \ref{the}]
Let $\uu^0 : \T^N \to \R^d$ be a Lipschitz continuous  initial data and $T>0$. Let $\uu(t,x) = S(t)\uu^0 (x)$, for $(t,x)\in [0,T]\times \T^N$.  For $n>0$ we define 
$\uu_n : [0,T] \times \T^N \to \R^d$ by 
$$\uu_n(t,x) =W_n(t)\uu^0(x) =  W(r) \circ W(T/2^n)^k \uu^0 (x),$$
where $kT/2^n +r = t$ and $0<r\leqslant T/2^n$, and $\uu_n(0,\cdot) = \uu^0$.  We will in fact prove that $\uu_n$ converges to $\uu$ as $n\to +\infty$.

We introduce the relaxed semi--limits, let us set $\underline \uu (t,x) = \liminf \uu_n(t_n,x_n)$ and $\overline \uu(t,x) = \limsup \uu_n(t_n,x_n)$ where the liminf and limsup are taken with respect to sequences $t_n \to t$ and $x_n \to x$.

Obviously, $\underline \uu \leqslant \overline \uu$. The core of the proof is to show that $\overline \uu$ (resp. $\underline \uu$) is a subsolution (resp. supersolution) of \eqref{intro evo wcs}. Proposition \ref{comparison} will then entail the reverse inequality, proving the convergence.

Let us prove that $\overline \uu$ is a subsolution, the proof for $\underline \uu$ being the same. Note that $\overline \uu$ is upper semi--continuous.
Let $i \in \ind$, $(t_0,x_0) \in (0,T) \times  \T^N$ and $\phi : [0,T) \times\T^N \to \R$ be a $C^1$ function such that $\overline u_i -\phi \leqslant 0$ attains a global strict maximum at $(t_0,x_0)$ by vanishing at this point. It follows there exists an extraction $m_n$ and points $(t_n,x_n)$ converging to $(t_0,x_0)$ such that $(\uu_{m_n})_i -\phi$ attains a global maximum at $(t_n,x_n) $ and such that $(\uu_{m_n})_i(t_n,x_n) \to \overline u_i(t_0,x_0)$. Denoting by $\xi_n = (\uu_{m_n})_i(t_n,x_n)  -\phi(t_n,x_n) $ we obtain that $\xi_n \to 0$ and that 
$(\uu_{m_n})_i \leqslant \phi +\xi_n$. Write $kT/2^{m_n} +r_n = t_n$ and $0<r_n\leqslant T/2^{m_n}$.

Let us fix an $\epsilon >0$ and construct a $C^1$ test function $\Phi : [0 , T]\times \T^N \to \R^d$ as follows: 
$\phi_i = \phi$, $\phi_j \geqslant \overline u_j +\eps/2  $ for $j\neq i$ and finally $\Phi(t_0,x_0) \leqslant \overline \uu(t_0,x_0)+\varepsilon\1$. Up to taking $n$ large enough, we still have the following property: $(\uu_{m_n})_j -\phi_j$ attains a global maximum at $(i,t_n,x_n) $.

We then compute
\begin{align*}
0 & =\frac{1}{r_n} \Big[(\uu_{m_n})(t_n,x_n) - W(r_n)\big[\uu_{m_n}(kT/2^{m_n}, \cdot)\big](x_n)\Big]_i \\
&\geqslant \frac{1}{r_n}\Big(\varphi (t_n,x_n) +\xi_n- W(r_n)\big[\Phi(kT/2^{m_n}, \cdot)+\xi_n\1\big]_i(x_n)\Big)\\
&= \frac{1}{r_n}\Big( \phi (t_n,x_n) -\phi(kT/2^{m_n}, x_n) +\phi(kT/2^{m_n}, x_n)  - W(r_n)\big[\Phi(kT/2^{m_n}, \cdot)\big]_i(x_n)\Big)\\
&\quad \qquad \qquad+  \frac{\xi_n}{r_n}\Big( 1 - (\e^{-r_nB}\1)_i\Big) 
\end{align*}
As $r_n, \xi_n\to 0$ the last term $ \frac{\xi_n}{r_n}\Big( 1 - (\e^{-r_nB}\1)_i\Big)$ converges to $0$ as $n\to +\infty$.

By making use of Proposition \ref{consistency} and letting $n\to +\infty$ we infer that 
\begin{align*}
0&\geqslant \frac{\partial \phi}{\partial t}(t_0,x_0)+ H_i\big(x_0, D_x \phi( t_0,x_0)\big)+[B\Phi(t_0,x_0)]_i     \\
& \geqslant \frac{\partial \phi}{\partial t}(t_0,x_0)+ H_i\big(x_0, D_x \phi( t_0,x_0)\big)+[B\uu(t_0,x_0)]_i  -\varepsilon \sum_{j\neq i} |b_{ij}|.
\end{align*}

Letting $\varepsilon \to 0$ shows that $\overline \uu$ is a subsolution.
\end{proof}

\section{Further properties and extensions of Theorem \ref{the}}
In this final section, we discuss some nice properties of the twisted operators $W$. Then we show how to weaken the hypotheses of our main theorem and propose some possible variations.
\subsection{Properties of $W$}\label{further}

\begin{prop}\label{majoration sous-sol}
If a  function $\uu : [0,T]\times\T^N \to \R^d$ is a   Lipschitz subsolution of the evolutionary equation then for any $t\leqslant T$ and any absolutely continuous curve $\gamma : [-t,0] \to \T^N$ the following holds:
$$\uu\big(t,\gamma(t)\big) -\e^{-tB}\uu\big(0,\gamma(0)\big) \leqslant \int_{0}^t \e^{(s-t)B} \L\big(\gamma(s),\dot\gamma(s)\big) \dd s.$$

In particular, $\uu\big( t , \gamma(t) \big) \leqslant W(t)\uu(0,\cdot)$. More generally, for any positive integer $n$ and positive times $t_1,\cdots , t_n$ such that $\sum t_i = t$, 
$$\uu\big( t , \gamma(t) \big) \leqslant W(t_n)\circ \cdots \circ W(t_1)\uu(0,\cdot).$$
\end{prop}
\begin{dimo}
Assume that $\uu$ is differentiable almost everywhere on the image of $\gamma$, then
\begin{align*}
\uu\big(t,\gamma(t)\big) -\e^{-tB}\uu\big(0,\gamma(0)\big)&= \int_0^t\frac{\dd}{\dd s} \e^{B(s-t)}\uu \big(s ,\gamma(s)\big)\dd s \\
&= \int_0^t \e^{(s-t)B}\Big[  B\uu\big(s,\gamma(s)\big)  +\frac{\partial \uu}{\partial t} \big(s,\gamma(s)\big)\\
&\qquad\qquad\qquad\qquad +D_x \uu\big(s,\gamma(s)\big)\cdot \dot \gamma(s)     \Big]\dd s   \\
&\leqslant  \int_0^t \e^{(s-t)B}\Big[  B\uu\big(s,\gamma(s)\big) + \frac{\partial \uu}{\partial t} \big(s,\gamma(s)\big)\\
&\qquad\qquad\qquad\qquad +\HH\big( \gamma(s),D_x \uu\big(s,\gamma(s)\big)\big)+\L\big(\gamma(s), \dot \gamma(s)   \big)  \Big]\dd s \\
&\leqslant  \int_{0}^t \e^{(s-t)B} \L\big(\gamma(s),\dot\gamma(s)\big) \dd s.
\end{align*}
Note that for the last inequality, we use the fact that all entries of the matrices  $\e^{(s-t)B}$ are non negative. The general case is then proved by an approximation argument of $\gamma$ by curves on which $\uu$ is differentiable almost everywhere.

 The second point is then the result of a straightforward induction on $n$.
\end{dimo}
\begin{oss}
It can actually be proved that the converse is also true in the above Proposition (see \cite{DSZ}).
\end{oss}

\begin{prop}
Let $\uu : \T^N\to \R^d$ then, for all $s,t\geqslant 0$,
$$W(s+t)\uu \geqslant W(s)\circ W(t) \uu.$$
\end{prop}
\begin{dimo}
Let $\gamma$ be a curve realizing the infimum for the first component of $W(s+t)\uu(x)$. We then have 
\begin{align*}
[W(s+t)\uu(x)]_1 &= \Big[\e^{-(t+s)B}\uu\big(\gamma(-(t+s))\big)+\int_{-(t+s)}^0 \e^{\sigma B} \L\big(\gamma(\sigma ),\dot\gamma(\sigma )\big) \dd \sigma\Big]_1\\
&=\Big[\e^{-sB}\Big(\e^{-tB} \uu\big(\gamma(-(t+s))\big)+\int_{-t}^0 \e^{\sigma B} \L\big(\gamma(\sigma-s ),\dot\gamma(\sigma-s )\big) \dd \sigma\Big) \\
&\quad\qquad\qquad\qquad\qquad\qquad \qquad +\int_{-s}^0 \e^{\sigma B} \L\big(\gamma(\sigma ),\dot\gamma(\sigma )\big) \dd \sigma \Big]_1 \\
&\geqslant \Big[\e^{-sB} W(t)\uu\big(\gamma(-s)\big) + \int_{-s}^0 \e^{\sigma B} \L\big(\gamma(\sigma ),\dot\gamma(\sigma )\big) \dd \sigma \Big]_1 \\
&\geqslant [W(s)\circ W(t) \uu]_1.
\end{align*}
\end{dimo}

Notice that in the previous inequality, there is no hope to obtain an equality, for the curve $\gamma$ has no reason to realize the infimum in $W(s+t)\uu(x)$ on other coordinates than the first one.

\begin{cor}
The sequence $W_n(t)\uu$  is decreasing with $n$. 
\end{cor}
\begin{dimo}
Using notations of the corollary, let $t = kT/2^n + s = k'T/2^{n+1} + s'$. Either $s<T/2^{n+1}$, then $k' = 2k$ and $s=s'$, or $T/2^{n+1}\leqslant s<T/2^{n}$, then $k' = 2k+1$ and $s-T/2^{n+1}=s'$. Let us deal with the first case, the second one being similar.
$$W_{n+1}(t)\uu = W^s\circ W(T/2^{n+1})^{2k}\uu = W^s\circ\Big(W(T/2^{n+1})^2\Big)^{k}\uu\leqslant W^s\circ W(t/2^n)^{k}\uu = W_n(t) \uu.$$
Moreover, by proposition \ref{majoration sous-sol} it is greater than $S(t)\uu$.
\end{dimo}

\begin{oss}\label{remarque}
\begin{enumerate}
\item
The previous results explain our choice of subdivision of the interval $[0,T]$ in our construction of $W_n$. Indeed, Theorem \ref{the} holds true for any sequence of partitions such that the length of the subdivisions uniformly converge to $0$. However, taking nested partitions (as we did) gives a decreasing family of operators. 
\item The corollary, along with Proposition \ref{majoration sous-sol} immediately imply that $W_n \uu$ converges given a Lipschitz function $\uu$. One alternative idea of proof would then be to establish that the limit is itself a subsolution. It would hence be the solution by maximality. However, to do so, we need to be able to keep track of the Lipschitz constants of the $W_n\uu$ which we were not able to do without requiring much stronger hypotheses on the Lagrangians. 
\end{enumerate}

\end{oss}

\subsection{Weakening the hypotheses of Theorem \ref{the}}\label{weakening}

The following Theorem weakens the Lipschitz hypothesis on the initial data $\uu^0$:

\begin{teorema}\label{C^0}
Assume $\uu^0 : \T^N \to \R^d$ is a continuous function, then the sequence $ W(t/2^n)^{2^n}\uu^0$ converges to $S(t)\uu_0$.
\end{teorema}

\begin{proof}
All the operators $S$, $W$ and $W_n$ are  $1$-Lipschitz, hence approximating $\uu^0$ by Lipschitz functions and using Theorem \ref{the} yields that again, 
$$ W(t/2^n)^{2^n}\uu^0 \to S(t)\uu^0.$$

\end{proof}

Finally, we show how to weaken the hypotheses on the Lagrangians:

\begin{teorema}\label{L_ilip}
Assume that the Lagrangians $L_i$ are Lipschitz continuous, convex in the $p$ variable and that there exists a Nagumo function $\theta$ verifying \eqref{nagumo} for which the $L_i$ satisfy \eqref{L2} in the almost everywhere sense.

Then the conclusions of Theorems \ref{the} and \ref{C^0} still hold.
\end{teorema}
\begin{proof}
The proof follows from the following simple observations. Assume that $\widetilde \L=(\widetilde L_i)_i$ are other Lagrangians such that $\|L_i - \widetilde L_i\|_\infty \leqslant \varepsilon$ for all $1\leqslant i\leqslant d$. We infer, by monotonicity of the Legendre transform, that $\|H_i - \widetilde H_i\|_\infty\leqslant \varepsilon$ for all $1\leqslant i\leqslant d$ (with obvious notations).

We then observe that if $\uu$ is an initial data, then $(t,x)\mapsto \widetilde S(t)\uu(x)-\varepsilon t\1$ is a subsolution for the weakly coupled system \eqref{intro evo wcs}. Hence, by the comparison principle, we conclude that $ \widetilde S(t)\uu-\varepsilon t\1 \leqslant S(t)\uu$. By a symmetric argument, we infer that $\|S(t)\uu - \widetilde S(t)\uu \|_\infty \leqslant \varepsilon t$. 

Moreover, using the explicit formulas defining $W$, $W_n$ and there analogues for $\widetilde \L$ denoted by $\widetilde W$ and $\widetilde W_n$ we obtain that for any continuous initial data $\uu$, we have as well that $\|W(t) \uu - \widetilde  W(t) \uu\| _\infty\leqslant \varepsilon t$ and $\|W_n(t) \uu - \widetilde  W_n(t) \uu\|_\infty \leqslant \varepsilon t$.

Hence, approximating uniformly $\L$ by strictly convex, smooth Lagrangians verifying the hypotheses of theorems \ref{the} and \ref{C^0} gives the result.
\end{proof}

\subsection{An alternative approximation scheme}

We conclude this section by proposing another way of approximating solutions to the weakly coupled system. The proofs being similar (even simpler with some respect) we omit them and leave them as an exercise to the motivated reader.

As in some respect, the structure of the schemes below are more simple, the proofs also work in the case of coupling matrices depending on the space variable. We henceforth consider a continuous, coupling matrix valued function $B: \T^N \to \mathcal M_d(\R)$.

\begin{definition}
Given a continuous initial condition $\uu^0$, we define the operator
$$
\mathcal W(t)\uu^0 (x) = \inf_{\substack{\gamma : [-t,0]\to \T^N \\ \gamma(0)=x }} \e^{-tB(x)}\uu^0\big(\gamma(-t)\big) + \int_{-t}^0  \L\big(\gamma(s),\dot\gamma(s)\big)\dd s.
$$
\end{definition}

\begin{teorema}\label{thebis}
Let $\uu^0:\T^N \to \R^d$ be a continuous function, then for any $t\geqslant 0$, the following holds:
$$S(t)\uu = \lim_{n\to +\infty} \mathcal W(t/2^n)^{2^n}\uu.$$
\end{teorema}

\begin{oss}
Actually, the proof of consistency of this scheme is easier and the hypotheses on the Nagumo functions and derivatives of the Lagrangians are not even needed. Indeed, as the Lagrangians appearing in the operator $\mathcal W$ are autonomous (contrarily to the ones in $W$ that depend on time because of the exponential term), conservation of energy gives the Lipschitz estimates on minimizing curves by only assuming that each $L_i$ is continuous, convex in $p$ and superlinear.

However, this operator is less natural and does not enjoy the nice properties established for $W$. And it was not misused in literature.
\end{oss}
The last scheme we propose consists in only taking the first terms of the exponential term: 

\begin{definition}
Given a continuous initial condition $\uu^0$, we define the operator
$$
\mathfrak W(t)\uu^0 (x) = \inf_{\substack{\gamma : [-t,0]\to \T^N \\ \gamma(0)=x }} \big({\rm Id}_d-tB(x)\big)\uu^0\big(\gamma(-t)\big) + \int_{-t}^0  \L\big(\gamma(s),\dot\gamma(s)\big)\dd s.
$$
\end{definition}

\begin{teorema}\label{thebis}
Let $\uu^0:\T^N \to \R^d$ be a continuous function, then for any $t\geqslant 0$, the following holds:
$$S(t)\uu = \lim_{n\to +\infty} \mathfrak W(t/2^n)^{2^n}\uu.$$
\end{teorema}

\appendix
\section{An explicit computation}
We conclude this article by giving a very simple example showing $W$ does not provides the viscosity solution operator. For sake of simplicity and of nice formulas, we consider here a problem on $\R^N$.

We will consider the simple system with $\mathbb{H}=\left(\begin{array}{c}H_1\\H_2 \end{array}\right)$ where 
$H_1=H_2=\frac12 \|\cdot \|^2$ on $\R^N$ and $B=\left(\begin{array}{cc}
1&-1\\
-1&1
\end{array}    \right)$. It then holds that 
$$\forall t\in \R, \quad \e^{tB}=\left(\begin{array}{cc}
\frac{1+\e^{2t}}{2}&\frac{1-\e^{2t}}{2}\\
\frac{1-\e^{2t}}{2}&\frac{1+\e^{2t}}{2}
\end{array}    \right).$$

We will also make use of that fact that if $H : \R^n\to \R$ is independent of the first variable and if $p\in \R^n$, then the solution to 
\begin{equation}\label{equ}
\frac{\partial u}{\partial t} +H(D_x u)=0
\end{equation}
with initial condition $u(0,x) = \langle p,x \rangle$ is given by 
\begin{equation}\label{formule simple}
u(t,x) =  -tH(p)+\langle p,x\rangle. 
\end{equation}
We now proceed to computing
\begin{align}
W(t)\uu^0 (x) &= \inf_{\substack{\gamma : [-t,0]\to \T^N \\ \gamma(0)=x }} \e^{-tB}\uu^0\big(\gamma(-t)\big) + \int_{-t}^0 \e^{sB} \L\big(\gamma(s),\dot\gamma(s)\big)\dd s \nonumber\\
&=\inf_{\substack{\gamma : [-t,0]\to \T^N \\ \gamma(0)=x }} \e^{-tB}\uu^0\big(\gamma(-t)\big) + \int_{-t}^0  \L\big(\gamma(s),\dot\gamma(s)\big)\dd s.  \label{formule}
\end{align}
As there is no exponential term in the integral, in formula \eqref{formule}, we recognize there a classical Lax--Oleinik formula and we can interpret that both lines of $W(t)\uu^0(x)$ are respectively solutions at time $t$ of the simple Hamilton--Jacobi equation \eqref{equ} with initial conditions, respectively the  entries of $\e^{-tB}\uu^0$.
 It follows that to compute $W(t)$ we have to compute the solution, at time $t$, of two classical Hamilton-Jacobi equations, with initial conditions given by the entries of $\e^{-tB}\uu^0$.

In our case, we take $\uu^0 (x) =\left( \begin{array}{c}0\\ \langle p,x\rangle\end{array}\right)$ therefore, $\e^{-tB}\uu^0(x) = \langle p,x\rangle\left( \begin{array}{c}\frac{1-\e^{2t}}{2}\\ \frac{1+\e^{2t}}{2}\end{array}\right)$. 

We deduce from \eqref{formule simple} that 
$$\uu(t,x):=W(t)\uu^0(x)=\frac{\langle p,x\rangle}{2}\left( \begin{array}{c}1-\e^{-2t}\\ 1+\e^{-2t}\end{array}\right)-\frac {t\|p\|^2}{8} \left( \begin{array}{c}(1-\e^{-2t})^2 \\ ( 1+\e^{-2t})^2 
\end{array}\right).$$
To conclude, we compute that
\begin{align*}
\frac{\partial \uu}{\partial t} +\mathbb H(D_x\uu)+B(\uu) &= \e^{-2t}\langle p,x\rangle \left( \begin{array}{c}1 \\- 1 
\end{array}\right) 
 -\frac {\|p\|^2}{8} \left( \begin{array}{c}(1-\e^{-2t})^2 \\ ( 1+\e^{-2t})^2 
\end{array}\right)  \\
&\qquad  +\frac {t\|p\|^2}{2} \left( \begin{array}{c}-(1-\e^{-2t})\e^{-2t} \\\e^{-2t} ( 1+\e^{-2t}) 
\end{array}\right)+ \frac {\|p\|^2}{8} \left( \begin{array}{c}(1-\e^{-2t})^2 \\ ( 1+\e^{-2t})^2 
\end{array}\right) \\
&\qquad 
+\frac{\langle p,x\rangle}{2}
\left(\begin{array}{cc}
1&-1\\
-1&1
\end{array}    \right)
\left( \begin{array}{c}1-\e^{-2t}\\ 1+\e^{-2t}\end{array}\right)\\
&\qquad -\frac {t\|p\|^2}{8}
\left(\begin{array}{cc}
1&-1\\
-1&1
\end{array}    \right)
\left( \begin{array}{c}(1-\e^{-2t})^2 \\ ( 1+\e^{-2t})^2 
\end{array}\right) \\
&=  \frac {t\|p\|^2\e^{-4t}}{2} \left( \begin{array}{c}1\\-1 \end{array}\right) \neq 0\1 .
\end{align*}
Hence $\uu$ is not a solution to the Hamilton--Jacobi system.

\bibliography{weakly}
\bibliographystyle{siam}

\end{document}